\documentclass{amsart}
\usepackage{amsmath,amssymb,amsthm,amsfonts,epsfig,verbatim}
\usepackage[all]{xy}

\renewcommand{\bar}[1]{\overline{#1}}
\makeatletter
\@addtoreset{equation}{section}
\renewcommand{\theequation}{\thesection.\@arabic\c@equation}
\newcommand\Nopagebreak{\@nobreaktrue\nopagebreak}

\newcommand{\com}[1]{
}

\newcommand{\us}[1]{{\upshape{#1}}}
\newcommand{\bm}{\begin{matrix}}
\newcommand{\enm}{\end{matrix}}
\newcommand{\bp}{\begin{pmatrix}}
\newcommand{\ep}{\end{pmatrix}}
\newcommand{\bsp}{\begin{split}}
\newcommand{\esp}{\end{split}}
\newcommand{\bc}{\begin{center}}
\newcommand{\ec}{\end{center}}
\newcommand{\ed}{{\rm d}}
\newcommand{\w}{{\mathchoice{\,{\scriptstyle\wedge}\,}{{\scriptstyle\wedge}}    
      {{\scriptscriptstyle\wedge}}{{\scriptscriptstyle\wedge}}}}                
  
\newcommand{\ol}{\overline}
\newcommand{\liealgebra}[1]{{\mathfrak {#1}}}
\newcommand{\End}{\liegroup{End}}
\newcommand{\gl}{\liealgebra{gl}}
\newcommand{\so}{\liealgebra{so}}
\newcommand{\lsl}{\liealgebra{sl}}

\newcommand{\csy}{\liealgebra{csp}}

\newcommand{\g}{\liealgebra{g}}
\newcommand{\h}{\liealgebra{h}}
\newcommand{\ft}{\liealgebra{t}}
\newcommand{\lk}{\liealgebra{k}}
\newcommand{\lp}{\liealgebra{p}}
\newcommand{\liegroup}[1]{{\operatorname{#1}}}
\newcommand{\G}{\liegroup{G}}
\newcommand{\K}{\liegroup{K}}

\newcommand{\GL}{\liegroup{GL}}
\newcommand{\SL}{\liegroup{SL}}
\newcommand{\SO}{\liegroup{SO}}

\newcommand{\Sp}{\liegroup{Sp}}
\newcommand{\CSp}{\liegroup{CSp}}
\newcommand{\Un}{\liegroup{U}}
\newcommand{\Or}{\liegroup{O}}

\newcommand{\GG}{\liegroup{G}_2}
\newcommand{\GGC}{{\mathbb G}_2}
\DeclareMathOperator{\stab}{Stab}

\newcommand{\Ad}{\liegroup{Ad}}
\DeclareMathOperator{\rank}{rank}
\renewcommand{\Re}{\operatorname{Re}}
\renewcommand{\Im}{\operatorname{Im}}

\newcommand{\om}{\omega}

\newcommand{\gr}{{\rm{Gr}}(2,7)}

\newcommand{\s}[1]{\mathbb{S}^{#1}}

\newcommand{\R}{\mathbb R}

\newcommand{\C}{\mathbb C}

\newcommand{\Z}{\mathbb Z}
\newcommand{\Oc}{\mathbb O}
\newcommand{\mca}{\mathcal A}
\newcommand{\mcb}{\mathcal B}

\newcommand{\mcl}{\mathcal L}

\newcommand{\mcc}{\mathcal C}
\newcommand{\mccb}{\bar{\mathcal C}}
\newcommand{\mcd}{\mathcal D}

\newcommand{\mcv}{\mathcal V}
\newcommand{\mcw}{\mathcal W}

\newcommand{\p}{\varphi}

\newcommand{\cp}[1]{{\C \mathbb{P}^{#1}}}

\newcommand{\al}{\alpha}

\newcommand{\ns}{\negthinspace}
\newcommand{\trp}{ {}^t \ns}
\newcommand{\hd}{{}^* \ns}

\newcommand{\la}{\lambda}

\newcommand{\rlg}{{\mcl}}
\newcommand{\rlgm}{{\mcl_-^\tau}}

\newcommand{\rlgmtw}{{\mcl_{-}^{\tau,\sigma}}}

\def\Id{\operatorname{Id}}

\usepackage[mathscr]{eucal}
\usepackage{bbm}
\usepackage{oldgerm}

\theoremstyle{plain}
\newtheorem{thm}{Theorem}[section]
\newtheorem{lem}[thm]{Lemma}

\newtheorem{cor}[thm]{Corollary}

\newtheorem{prop}[thm]{Proposition}

\newtheorem{rem}[thm]{Remark}
\theoremstyle{definition}
\newtheorem{defn}{Definition}[section]

\newtheorem{exam}[thm]{Example}
\newtheorem{exams}[thm]{Examples}

\newcommand{\be}{\begin{equation}}
\newcommand{\ee}{\end{equation}}
\newcommand{\bs}{\begin{section}}
\newcommand{\es}{\end{section}}
\newcommand{\bss}{\begin{subsection}}
\newcommand{\ess}{\end{subsection}}
\newcommand{\bit}{\begin{itemize}}
\newcommand{\eit}{\end{itemize}}

\def\pr{\mathbb{P}}
\def\Aut{\operatorname{Aut}}

\begin{document}
\title[Rational Loop Groups]{Generators for Rational Loop Groups and Geometric Applications}
\author{Neil Donaldson  }
\author{Daniel Fox}

\author{Oliver Goertsches}

\address{Department of Mathematics\\ University of California Irvine\\ Irvine, California 92697}
\email{ndonalds@math.uci.edu, dfox@math.uci.edu, ogoertsc@math.uci.edu}

\thanks{The authors would like to thank Chuu-Lian Terng for useful conversations and for suggesting the problems addressed in this article.
The third author was supported by the Max-Planck-Institut f\"ur Mathematik in Bonn and a DAAD-postdoctoral scholarship.}

\date\today
\keywords{Loop group, integrable system, submanifold geometry, generator theorem, simple element, neo-classical}

\begin{abstract} Uhlenbeck proved that a set of simple elements generates the group of rational loops in $\GL(n,\C)$ that satisfy the $\Un(n)$-reality condition.   For an arbitrary complex reductive group, a choice of representation defines a notion of rationality and enables us to write down a natural set of simple elements.  Using these simple elements we prove generator theorems for the fundamental representations of the remaining neo-classical groups and most of their symmetric spaces. In order to apply our theorems to submanifold geometry we also obtain explicit dressing and permutability formulae. We introduce a new submanifold geometry associated to $\GG/\SO(4)$ to which our theory applies.
\end{abstract}

\maketitle

\tableofcontents

\bs{Introduction}
The aim of this paper is threefold:  
\begin{itemize}
\item to provide a general definition of simple elements, motivated by the observation that the notion of rationality of holomorphic loops depends on a choice of representation
\item to prove generator theorems for (twisted) rational loop groups using these simple elements
\item to introduce a new submanifold theory that is governed by the $\GG/\SO(4)$-system.  
\end{itemize}

Generating new solutions of geometric problems by means of holomorphic loop group dressing actions has a long history. The abstract theory was formalised by Pressley--Segal \cite{Pressley1986}; in contrast some of the details of the widespread geometric applications are somewhat colloquial.
Given a complex reductive Lie group $G$, the (holomorphic) loops into $G$ are holomorphic maps $g$ from (open dense) subsets of $\C\pr^1$ into $G$. As is well known, the Birkhoff factorisation theorem yields a dressing action of negative loops $g_-\in \mcl_-(G)$ (holomorphic near $\infty$ with $g_-(\infty)=1$) on positive loops $g_+\in \mcl_+(G)$ (holomorphic on $\C$) as follows: given generic $g_\pm\in \mcl_\pm(G)$, there exist $\hat g_\pm\in \mcl_\pm(G)$ such that $g_-g_+=\hat g_+\hat g_-$; the \emph{dressing action} of $g_-$ on $g_+$ is then defined by $g_- * g_+:=\hat g_+$.

Many examples of this set-up appear throughout the literature, often with additional (reality or twisting) conditions on the loops: for example the loops satisfying $\tau g(\lambda)=g(\bar\lambda)$ where $\tau$ is some anti-holomorphic involution of $G$ form the subgroup $\mcl_-^\tau(G)$ and are said to satisfy a \emph{reality condition}. In many examples smooth maps into the positive loop group are seen to correspond to special submanifolds, or solutions to a system of PDEs. These are preserved under dressing transforms by negative loops and so the transforms descend to transforms of submanifolds and solutions of PDEs. For example Uhlenbeck \cite{Uhlenbeck1989} applies this construction to harmonic maps into $\Un(n)$; Terng's group \cite{Terng2000,Bruck2002,Donaldson2008} study her $U/K$-systems \cite{Terng1997}; other applications include \cite{Burstall1994,Ferus1996a,Burstall2004}. Calculating and working explicitly with the dressing action of a generic negative loop is, in general, extremely difficult.  Therefore most of the authors cited above make use of special negative loops, termed \emph{simple factors}, or \emph{simple elements}, the dressing action of which is straightforward to compute.

Definitions of simple factors have been given in special situations, but thusfar no general definition has been forthcoming. Examples of simple factors typically exhibit several of the following desirable properties:
\begin{enumerate}
\item They are \emph{rational} loops with as few poles as possible with lowest possible degree.
\item Their dressing action is explicitly calculable.
\item They act by (simple) geometric transforms on any underlying geometry of the Lie group.
\item There is a permutability theorem: given simple factors $p_1,p_2$, there should exist simple factors $\hat p_1,\hat p_2$ such that $\hat p_1p_2=\hat p_2p_1$; the dressing transforms of each side therefore correspond.
\item The collection of simple factors generates the rational loop group.
\end{enumerate}
The above list describes the maximally desirable situation.  However, the fifth point has only ever been demonstrated for $\GL(n,\C)$ \cite{Uhlenbeck1989}\footnote{More recently, a unique factorisation theorem for the same group was obtained by Dai--Terng \cite{Dai2007}.} and for the twisted loop group associated to $\Un(n)/\Or(n)$ \cite{Terng2006}. While we have in mind the application of dressing to geometric problems, the main thrust of this paper is on finding generators for rational loop groups. The applications of such theorems are immediate: if one can factor any rational loop as a product of simple factors whose dressing action is easily computable, then, in principle, the action of any rational loop is computable.\\

As was noted by Pressley--Segal \cite{Pressley1986}, rationality of loops is only defined for matrix groups. More formally, in Section \ref{sec:rat}, we define rationality for any group $G$ with respect to any representation: $g\in \mcl(G)$ is rational with respect to a representation $\rho:G\to\End(V)$ iff $\rho\circ g$ is rational. In the literature, authors have generally used the adjoint representation or, when dealing with matrix groups, the standard matrix representation.

In Section \ref{sec:simpleelements} we define simple factors for any representation of a complex reductive Lie group satisfying the reality condition with respect to a compact real form. The simple factors depend on the chosen representation. The various definitions of a simple factor given in the papers above are all special cases of the definition presented in Section \ref{sec:simpleelements}.

In Sections 4, 5 \& 6 we prove generating theorems for the rational loop groups of the fundamental representations of $\SO(n,\C),\ \GGC$, and the conformal symplectic group $\CSp(n,\C)$. Together with Uhlenbeck's work this establishes generating theorems for the fundamental representations of all of the neo-classical groups. The appearance of the conformal symplectic group may seem strange at first, but it is a natural consequence of attempting to prove a generating theorem.  Since centres of groups act trivially while dressing, this is of no concern for geometric applications. Indeed this approach has been followed before: dressing of positive loops in $\SL(n,\C)$ tends to be done instead by $\GL(n,\C)$-simple factors; e.g.\ \cite{Terng2000}.

Our simple factors have the same general form in each situation: all are sums of projections onto certain weight spaces with rational functions as coefficients. For example, in $\CSp(n,\C)$ we calculate that
\[p_{\alpha,W}(\lambda)=\frac{\lambda-\alpha}{\lambda-\bar\alpha}\pi_W+\pi_{W^\perp},\]
is a simple element, where $\pi_W$ is projection onto a Lagrangian subspace $W\subset\C^{2n}$ and ${}^\perp$ is Hermitian perp. Although our generating theorems follow the same strategy that was used by Uhlenbeck in $\GL(n,\C)$ --- expand a rational loop in a power series about a pole and use simple elements to reduce the order of the pole --- the linear algebra associated to each group and representation alters the nature of the details.  The case of the fundamental representation of $\GG$ proved to be particularly subtle, and aside from cases such as $\SO(3)$, in which the adjoint representation is isomorphic to the standard representation, we do not have proofs for generating theorems for any representations other than the fundamental ones.

For the fundamental representation of each group we prove a permutability formula, and explicit expressions for the dressing action of simple factors. In contrast to the classical groups, and as observed in \cite{Burstall2004}, \cite{Burstall1997}, \cite{Pressley1986}, the dressing action and permutability formulae for $\GGC$ must depend on derivatives of the loop being dressed.

We extend our generating theorems to the twisted loop groups associated to all symmetric spaces with $\SO(n)$ or $\GG$ as their isometry group and the symmetric space $\CSp(n)/\Un(n)$.\footnote{We do not currently see how to handle the complex and quaternionic Grassmannians.}  This is what is required for applications to $U/K$-systems and harmonic maps into symmetric spaces.   

Terng and Wang \cite{Terng2006} introduced an affine extension of the $\Un(n)/\Or(n)$-system, allowing them to characterize flat Lagrangian submanifolds of $\C^n$ in terms of a twisted loop group.  By using an extended Gauss map we define the notion of a $\GG/\SO(4)$-abelian surface in $\R^7$ and show that such surfaces are equivalent to solutions of an affine extension of the $\GG/\SO(4)$-system.  The simple elements can be used to generate new solutions from any solution at hand.  
\es

\section{The rational loop group associated to a representation}\label{sec:rat}

Associated to every complex Lie group $G$ is its {\bf holomorphic loop group}
\[
\mcl(G):= \left\{ g:U\to G \text { holomorphic}\mid U\subset \cp{1} \text{ is open and dense} \right\};
\]
multiplication requires intersecting the respective domains. We do not distinguish between loops that agree on an open and dense set.  



We are interested in subgroups of $\mcl(G)$ consisting of functions that are rational on $\cp{1}$, but the notion of rationality is not defined when the target space is an arbitrary complex manifold.  Already, Pressley--Segal \cite{Pressley1986} point out that the notion of a rational loop group only exists for matrix groups.  So for any representation $\rho:G\to \GL(V)$, we can define the {\bf rational loop group associated to $\rho$} to be
\[
\rlg(G,\rho):=\left\{g\in \mcl(G)\mid \rho\circ g:\cp{1}\to  \End(V) \text{ is meromorphic}\right\}.
\]
The notation $\rlg(G,V)$ will also be used.

\begin{exams}\hspace{1cm}
\begin{enumerate}
\item Let $G=\C^*$ and $\rho_k:\C^*\to \GL(\C); z\mapsto z^k$. Then 
\[
\rlg(\C^*,\rho_k):=\left\{g:\cp{1}\to \C\mid g^k \text{ is rational}\right\}.
\]
\item $\rlg(G,1)=\mcl(G)$, where $1$ denotes the trivial representation.
\item $\rlg(G,\rho\oplus\sigma)=\rlg(G,\rho)\cap \rlg(G,\sigma)$ for all representations $\rho$ and $\sigma$.
\item Let $G=\SL(3,\C)$. Then the loop
\[
\la\mapsto \bp \la^{\frac{1}{3}}&0&0\\ 0&\la^{\frac{1}{3}}&0\\0&0&\la^{-\frac{2}{3}}\ep
\]
is rational with respect to the adjoint representation, but not with respect to the standard representation.
\end{enumerate}
\end{exams}

\begin{lem}\label{lem:ratlem}
The loop $f \in \mcl(G)$ is meromorphic at $\la_0$ with respect to $\rho$ if and only if for all $v \in V$,
\[
\la \mapsto \rho(f(\lambda)) v
\]
is meromorphic at $\lambda_0$.  In fact, it is enough to only consider a basis of $V$.
\end{lem}
Let $G$ be connected complex reductive and $\tau$ an antiholomorphic involution on $G$ whose fixed point group $G^\tau$ is a compact real form of $G$. A loop $g$ is said to satisfy the {\bf reality} condition with respect to $\tau$ if
\[
\tau(g(\bar \la))=g(\la).
\]
If $\sigma$ additionally is an holomorphic involution on $G$ commuting with $\tau$ and $K=G^{\tau,\sigma}$ denotes the fixed point set of both $\tau$ and $\sigma$, 
then $G^\tau/K$ is a symmetric space. We say that $g$ is {\bf twisted} with respect to $\sigma$ if 
\[
\sigma(g(-\la))=g(\la).
\]
A loop $g$ is {\bf negative} if it is holomorphic at $\infty$ and $g(\infty)=e$; it is {\bf positive} if it is holomorphic on all of $\C$.
We use superscripts for the reality and twisting conditions, and subscripts for negativity and positivity. As before, the representation will occur in the notation to indicate rationality. For example, the group of negative rational loops satisfying both reality and twisting is denoted by $\rlgmtw(G,\rho)$ or $\rlgmtw(G,V)$.

\section{Simple elements}\label{sec:simpleelements}
As before, let $G$ be a connected complex reductive Lie group, i.e.~the complexification of some compact real Lie group. In the following, we fix an antiholomorphic involution $\tau$ of $G$ giving rise to the compact real form $G^\tau$.    Let $\rho:G\to \GL(V)$ be a representation. We assume that $\rho$ is almost faithful, i.e.~has discrete kernel. Note that because $G$ is reductive, $\rho$ is completely reducible.  In this section we formulate the notion of a simple element in $\rlgm(G,\rho)$. 
\begin{defn}
A semisimple element $H \in \g$ is $\rho$-{\bf integral} if $\rho(H)\in \End(V)$ has only integer eigenvalues.
\end{defn} 
Fix a maximal torus $T\subset G^\tau$ with Lie algebra $\ft$ and denote the complexification of $\ft$ by $\h\subset \g$.  The elements of $\h$ are semisimple and commute and thus $\h$ induces the decomposition 
\be\label{eqn:weightdecomp}
V=\bigoplus_{\mu\in \Phi_{\rho}} V_\mu
\ee
into simultaneous eigenspaces.  For any $\mu \in \h^*$ the subspace $V_{\mu} \subset V$ consists of those vectors $v$ such that $\rho(H)v=\mu(H)v$ for all $H \in \h$.  The indexing set $\Phi_{\rho} \subset \h^*$ is the set of {\bf weights}, i.e.~those $\mu$ for which $V_\mu$ is nontrivial.

Recall that because $G^\tau$ is compact, there is a $G^\tau$-invariant hermitian inner product on $V$, with respect to which the elements of $\g^\tau$ become skew-hermitian operators on $V$. Consequently, $\rho$-integral elements in $\h$ are necessarily contained in the space $i\ft\subset \g$. Note that $H\in \h$ is $\rho$-integral if and only if $\mu(H)\in \Z$ for all weights $\mu$. Using $\rho$-integral elements we can define simple elements, depending on the chosen representation:
\begin{defn}
For any $\alpha \in \C \setminus \R$ and $\rho$-integral $H\in i\g^\tau$, the loop 
\[
p_{\alpha,H}(\lambda)=\exp\left(\ln\left(\frac{\la-\alpha}{\la - \bar\alpha}\right)H\right),
\]
where $\ln$ denotes the natural logarithm, is a {\bf simple element.} 
\end{defn} 

The assumption that $H\in i\g^\tau$ is needed since we want our simple elements to satisfy the $\tau$-reality condition:
\begin{lem}\label{lem:reality}
For any $\rho$-integral $H\in i\g^\tau$ and $\alpha\in \C\setminus \R$, the simple element
$p_{\alpha,H}$
satisfies the reality and normalization conditions.
\end{lem}

\begin{proof} Simple elements are obviously normalized at infinity. The reality condition is satisfied because
\begin{align*} 
\tau (p_{\al,H}(\bar{\lambda})) &=\tau\exp\left(\ln\left(\frac{\bar\la-\alpha}{\bar\la - \bar\alpha}\right)H\right)
=\exp\left(\ol{\ln\left(\frac{\bar\la-\alpha}{\bar\la - \bar\alpha}\right)}\tau(H)\right)\\
&= \exp\left(-\ln\left(\frac{\la-\bar\alpha}{\la -\alpha}\right) H\right)= p_{\al,H}(\lambda),
\end{align*}
where we used that $\tau(H)=-H$.
\end{proof}
For the following lemma, let $\sigma$ be a holomorphic involution commuting with $\tau$ and $\g^\tau=\lk\oplus \lp$ the corresponding Cartan decomposition, that is, $\lk$ is the $+1$-eigenspace of $\sigma$ and $\lp$ is the $-1$-eigenspace.
\begin{lem}\label{lem:twisting}
For any $\rho$-integral element $H\in i\lp$ and $r\in \R\setminus\{0\}$, the simple element $p_{ir,H}$ satisfies the $\sigma$-twisting condition.
\end{lem}
\begin{proof} Using $\sigma(H)=-H$ we calculate
\[\sigma (p_{ir,H}(-\la))=\exp\left(\ln\left(\frac{-\la-ir}{-\la + ir}\right)\sigma(H)\right)=\exp\left(-\ln\left(\frac{\la+ir}{\la - ir}\right)H\right)=p_{ir,H}(\la).
\]
\end{proof}
Simple elements with $H\in i\ft$ act diagonally with respect to the $\ft$-weight decomposition:

\begin{prop}\label{prop:simple} Let $H\in i\ft$ be $\rho$-integral. 
The action of the simple element $p_{\alpha,H}$ is given by
\[
\rho \circ p_{\alpha,H}(\la) = \sum_{\mu \in \Phi_{\rho}} \left( \frac{\la -\alpha}{\la - \bar\alpha} \right)^{\mu(H)} \pi_{V_{\mu}}  
\]
where the $\pi_{V_{\mu}}$ are projections with respect to the decomposition \eqref{eqn:weightdecomp}.
\end{prop}
\begin{proof}
For any weight vector $v \in V_{\mu}$, 
\begin{align*}
p_{\alpha,H}(\la) \cdot v&= \exp\left(\ln\left(\frac{\la-\alpha}{\la - \bar\alpha}\right)H\right)v \\
&= \exp\left(\ln\left(\frac{\la-\alpha}{\la - \bar\alpha}\right) \mu(H) \right)v=\left( \frac{\la -\alpha}{\la - \bar\alpha} \right)^{\mu(H)} v.
\end{align*}
By Lemma \ref{lem:ratlem}, the loop $p_{\alpha,H}$ is rational with respect to $\rho$ because $\mu(H) \in \Z$.
\end{proof}

Any $\rho$-integral element $H\in i\g^\tau$ is $G^\tau$-conjugate to one in $i\ft$.  In particular, the proposition shows that simple elements are rational, i.e.~$p_{\al,H}\in \rlgm(G,\rho)$.  Since $\rho$ is almost faithful, the $\rho$-integral elements in $i\ft$ 
form a lattice in $i\ft$. If we choose an integral basis $H_1,\ldots,H_r$ for this lattice, any simple element of the form $p_{\alpha,H}$ with $H\in i\ft$ is a product of the $p_{\alpha,H_i}$, since $\exp$ is a homomorphism on $\h$.

\begin{exam} \label{exam:un}
Let $G=\GL(n,\C)$. The automorphism 
\[
\tau(g)=\hd g^{-1}.
\]
(where $\hd g = \trp \bar{g}$) gives rise to the compact real form $G^\tau=\Un(n)$.
Choose the maximal torus consisting of diagonal elements.  The weights are $\mu_1, \ldots, \mu_n$ where
\[
\mu_i(X)=X_{ii}.
\]
for any diagonal matrix $X$.  The corresponding weight spaces are $L_i=\C \cdot e_i$ where $e_1,\ldots,e_n$ are the standard basis vectors of $\C^{n}$.  Let $H_i$ be the diagonal matrices dual to the $\mu_i$.  The $\{H_i\}$ form a basis of the lattice of $\rho$-integral elements.  From Proposition \ref{prop:simple} it follows that
\be\label{eqn:unsimple}
\rho \circ \exp\left(\ln\left(\frac{\la-\alpha}{\la - \bar\alpha}\right) H_i \right) =\left(\frac{\la-\alpha}{\la - \bar\alpha} \right) \pi_{L_i}+\pi_{K_i}=\Id+\frac{\bar\alpha - \alpha}{\la - \bar\alpha}\pi_{L_i}
\ee
where $K_i=L_1 \oplus \ldots \oplus \hat{L_i}\oplus \ldots  \oplus L_n$, and the projections are defined with respect to the hermitian orthogonal decomposition of $\C^n$ into the $L_j$. Conjugating these with $\Un(n)$, we see that the loops of the form
\[
\Id+ \frac{\bar\alpha -\alpha}{\la - \bar\alpha}\pi_{L},
\]
where $L\subset \C^n$ is any complex line and $\pi_L$ is the hermitian projection onto $L$, are simple elements. 
By multiplying several of these for which the complex lines are orthogonal, we obtain loops of the form 
\[
\Id + \frac{\bar\alpha -\alpha}{\la - \bar\alpha}\pi_{W},
\]
where $\pi_W$ is the hermitian projection onto a complex subspace $W$; 
these are the simple elements introduced in \cite{Uhlenbeck1989}.
\end{exam}
\begin{exam}
In the case that $\rho$ is the adjoint representation of a complex Lie group $G$ admitting parabolic subalgebras of height one, our simple elements generalise those considered in \cite{Donaldson2006}. Grading elements, in particular canonical elements of pairs of complementary parabolic subalgebras, are $\Ad$-integral.
\end{exam}
\begin{exam}
Consider a pair of representations $\rho_1:G \to \GL(V_1)$ and $\rho_2:G \to \GL(V_2)$ such that $0$ is a weight for $V_1$. Then the simple elements for $V_1 \otimes V_2$ are a subset of those for $V_2$. 
\end{exam}
\begin{exam}
Consider the case of $\SL(2,\C)$.  There is an irreducible representation $V_n$ of dimension $n$ and $V_n \subset Sym^{n-1}(\C^2)$.  In this case the simple elements only depend on whether $n$ is odd or even.
\end{exam}




We now turn our attention to the question of finding generators for the rational loop groups.
In Sections \ref{sec:genso} and \ref{sec:geng2} we will prove that the simple elements generate the rational loop groups $\rlgm(\SO(n,\C),\C^n)$ and $\rlgm(\GGC,\C^7)$. We think that it is plausible that the simple elements generate the negative rational loop group whenever $\rho$ is the complexification of an orthogonal representation, i.e.~$V=W\otimes \C$, $\rho(G^\tau)\subset \Or(W)$ (and consequently $\rho(G)\subset \Or(V)$). 
However, if the representation is not orthogonal, we believe the same statement will only be true if we pass to the extension by the center of $\GL(V)$, i.e.~the conformalization of $G$. An example of this will be given in Section \ref{sec:gencsp}. There we show that the simple elements generate $\rlgm(\CSp(n,\C),\C^{2n})$, where $\CSp(n)$ is the conformalization of $\Sp(n,\C)$.

\section{$\CSp(n,\C)$}

\subsection{Generating $\rlgm(\CSp(n,\C),\C^{2n})$} \label{sec:gencsp}
In this section, we prove that for the standard representation of $\CSp(n,\C)$, which is $\C^{2n}$ by our indexing, the simple elements defined in Section \ref{sec:simpleelements} generate the rational loop group $\rlgm(\CSp(n,\C),\C^{2n})$.  

Let $\omega$
be the standard symplectic form on $\C^{2n}$.    If
\[ J=\bp 0 & I_n \\ -I_n & 0 \ep\] 
then 
\[  \om(v,w)=\trp v J w.\]
The simple group $\Sp(n,\C)$ is the group of (invertible) matrices that preserve $\om$ and the conformalization of $\Sp(n,\C)$ is the extended group
\[
\CSp(n,\C)=\{A\in \GL(2n,\C)\mid A^*\omega=c\cdot \omega \text{ for some }c\in \C\}.
\]
We will make use of the standard hermitian inner product $(v,w)=\trp \bar{v} w$ and the relation 
\[
(v,w)=\omega(\ol{v},Jw).
\]
The Lie algebra of $\CSp(n,\C)$ is 
\[
\csy(n,\C)=\left\{\left. \bp A & B \\ C & -\trp A \ep + a\cdot \Id\ \right|\  \trp B=B, \trp C=C, a\in \C \right\}.
\]
The involution $\tau(g)=\hd g^{-1}$ of $\CSp(n,\C)$ gives rise to the compact real form $\CSp(n)=\Sp(n)\times S^1$.
It contains the $n+1$-dimensional torus consisting of diagonal matrices in $\CSp(n)$ whose complexification has the Lie algebra
\[
\h = \left\{ \left. \bp A & 0 \\ 0 & -A\ep + a\cdot \Id\ \right| \ A \text{ is diagonal}, a\in \C\right\}.
\]
The weights of the standard representation of $\CSp(n,\C)$ on $\C^{2n}$ are $\mu_1,\ldots,\mu_{2n}$, where $\mu_i(X)=X_{ii}$, and the corresponding weight spaces are spanned by the standard basis elements $e_1,\ldots,e_{2n}$.

Let $H=\bp I_n & 0 \\ 0 & 0\ep \in \h$. By Propositition \ref{prop:simple}, the simple element corresponding to the $\rho$-integral element $H$ is given by
\[
\exp\left(\ln\left(\frac{\la-\al}{\la-\ol{\al}}\right)H\right)=\left(\frac{\la-\al}{\la-\ol{\al}}\right)\pi_{W_0} + \pi_{W_0^\perp},
\]
where $W_0=e_1\w \ldots \w e_n$ is Lagrangian. Conjugating this with $\CSp(n)$, we obtain that for any Lagrangian subspace $W\subset \C^{2n}$, the loop
\[
p_{\al,W}(\la)=\left(\frac{\la-\al}{\la-\ol{\al}}\right)\pi_W + \pi_{W^\perp},
\]
is a simple element. Note that the hermitian complement $W^\perp$ of $W$ is $J\ol{W}$. 
\begin{thm} \label{thm:gencsp} The simple elements $p_{\al,W}$, where $\al\in \C\setminus \R$ and $W\subset \C^{2n}$ is a Lagrangian subspace, generate $\rlgm(\CSp(n,\C),\C^{2n})$.
\end{thm}
\begin{proof} Let $g\in \rlgm(\CSp(n,\C),\C^{2n})$ and fix a pole $\al\in \C\setminus \R$.   We  write the Laurant expansion of $g$ in  $\frac{\la-\al}{\la-\bar{\alpha}}$ explicitly as
\be
g(\la)=\sum_{j=-k}^\infty \left(\frac{\la-\al}{\la-\ol{\alpha}}\right)^j g_j
\ee
 with $g_{-k}\neq 0$.  The total degree of the pole at $\al$ is defined to be the pair $(k,\rank g_{-k})$.  We define $(k,n) < (l,m)$ if $k<l$ or $k=l$ and $n<m$.  We proceed using induction on the total degree of each of the finitely many poles of the loop $g$.  The reality condition $\tau(g(\ol{\la}))=g(\la)$ implies that if $g$ is holomorphic at $\al$ (i.e.~ it has no pole and is group valued) then it is also holomorphic at $\ol{\al}$. Our approach is to first remove the pole at $\al$ and then to modify the resulting loop so that it is holomorphic at $\al$ and thus also at $\ol{\al}$.

The fact that $g$ is a map into $\CSp(n,\C)$ means that there is a $\C$-valued map $c_{g}$ satisfying
\be\label{eq:cspcond}
\omega(g(\la)v,g(\la)w)=c_{g}(\la)\cdot\omega(v,w)
\ee
for all $v,w\in \C^{2n}$. 

Inserting the power series expansion of $g$ into \eqref{eq:cspcond}, we see that $c_{g}$ cannot have a pole at $\al$ of order higher than $2k$. If $c_{g}$ has a pole of order exactly $2k$, we get
\[
\omega(g_{-k}v,g_{-k}w)=c_0\cdot \omega(v,w)
\]
for some complex number $c_0\neq 0$. In particular, it follows that $g_{-k}$ is invertible, whereupon we form
\[
\tilde{g}=p_{\al,W} g
\]
for an arbitrary Lagrangian subspace $W$; since the new $-k$-coefficient then has image only $W^\perp$, this reduces the total degree of the pole.

If $c$ has a pole of lower order at $\al$ (or none at all), \eqref{eq:cspcond} reduces to 
\[
\omega(g_{-k}v,g_{-k}w)=0,
\]
i.e.~$\Im g_{-k}$ is $\omega$-isotropic. If we let $W$ be a Lagrangian subspace containing $\Im g_{-k}$, then the product $p_{\al,W}g$ has a pole at $0$ of order at most $k-1$.

We continue this process until we obtain a loop (still referred to as $g$) that no longer has a pole at $\al$  and now show that it can be modified by simple factors so as to make it holomorphic at $\al$. 

If $g_0=0$, we can certainly multiply with elements of the form $p_{\al,W}^{-1}$ until $g_0\neq 0$, so we 
may assume $g_0\neq 0$. As above, condition \eqref{eq:cspcond} then gives that $g_0$ is either invertible 
(in which case we are done) or that its image is $\omega$-isotropic. 

In the second case, consider  the map $\la\mapsto \det(g(\la))$. It is a polynomial in $\frac{\la-\al}{\la-\ol{\al}}$. Denote by $m$ the order of its zero at $\al$.
Let $W$ be a Lagrangian subspace such that $\Im g_0\subset W^\perp$ and define
$\tilde{g}=p_{\al,W}^{-1}g$. Then, $\tilde{g}$ has no pole at $\al$, its value $\tilde{g}(\al)$ is not zero, and the function
\[ \la\mapsto \det(\tilde{g}(\la))=\left(\frac{\la-\al}{\la-\ol{\al}}\right)^{-n} \det(g(\la))\] has a zero at $\al$ of order $m-n$. By induction, we can completely remove this zero and are left with a loop $g$ holomorphic at $\alpha$.

Once all poles are removed, we are left with a pole-free loop on $\cp{1}$ with value $\Id$ at $\infty$. Liouville's theorem implies that this is the identity loop, and we are finished: $g$ has been written as a product of the generators.
\end{proof}

\subsection{Dressing and permutability}

The following theorem is well known for the groups $\GL(n,\C)$ \cite{Uhlenbeck1989} and $\SO(n,\C)$ \cite{Burstall2004}.  We present it in order to be self-contained.
\begin{thm}\label{thm:spdressing}
Let $h \in \mcl(\CSp(n,\C))$ be holomorphic at $\al \in  \C \setminus \R$.  If we define $W'=h(\al)^{-1}W$ then
\[
p_{\al,W}hp^{-1}_{\al,W'}
\]
is again holomorphic at $\al$.
\end{thm}
\begin{proof}
The compact group $\CSp(n)$ preserves the isotropic condition, so $W'$ is also Lagrangian, and $p_{\al,W'}$ is again a simple factor.  We expand the new loop at $\al$
\[
p_{\al,W}hp^{-1}_{\al,W'}(\la)=\left(\frac{\la - \ol{\al}}{\la - \al}\right)\pi_{W^{\perp}}h(\al)\pi_{W'}+\ldots
\] 
By our choice of $W'$ the pole part vanishes, so it makes sense to evaluate the loop at $\al$. Since $\CSp(n,\C)$ is closed in $\GL(n,\C)$, it only remains to show that the endomorphism of $\C^{2n}$ obtained by evaluating it at $\al$ is invertible. But its determinant is the same as the determinant of $h(\al)$, which is nonzero.
\end{proof}
\begin{cor}
Suppose that $p_{\al,W}$ and $p_{\beta,V}$ are simple factors such that $\al\neq \beta,\bar\beta$. Then
\[
p_{\beta,p_{\al,W}(\beta)V}p_{\al,W}=p_{\al,p_{\beta,V}(\al)W}p_{\beta,V}.
\]
\end{cor}
This follows from Theorem \ref{thm:spdressing} in the standard way. 
 \subsection{Generators for a family of twisted loop groups}

Let $\sigma$ be the holomorphic involution of $\Sp(n,\C)$ defined by $\sigma(A)=JAJ^{-1}$, where 
 \[
 J=\left(\begin{matrix}0 & I_n \\ -I_n & 0 \end{matrix}\right),
 \] 
 giving rise to the symmetric space $\Sp(n)/\Un(n)$. Extend $\sigma$  to an involution $\sigma_c$ on $\CSp(n,\C)$ that acts on the center by $D \mapsto D^{-1}$. Also denote the derivative at the identity by $\sigma_c$. While the fixed point set of $\sigma_c$ is again $\Un(n)$, so that the respective symmetric spaces are different, the choice of a non-trivial extension of $\sigma$ aids the generator theorem. Recall that we say a loop $g$ is {\bf twisted}, or satisfies the {\bf twisting condition}, if
 \[\sigma_c(g(-\lambda))=g(\lambda).\]
 
 \begin{lem}
 If $\alpha\notin i\R$, and $p_{\alpha,W}$ is a simple element, then the product
 \[q_{\alpha,W}:=
 p_{-\alpha,p_{\alpha,W}(-\alpha)s W}p_{\alpha,W}= p_{\alpha,p_{-\alpha,s W}(\alpha)W}p_{-\alpha,sW}\]
 satisfies the twisting condition.
 \end{lem}
 
 The proof follows immediately from the permutability proposition.
 
 \begin{thm}
 The elements
 \begin{gather*}
 p_{\alpha,W}(\lambda)=\left(\frac{\la-\al}{\la-\ol{\al}}\right)\pi_W + \pi_{W^\perp},\ \text{where }W\text{is a real Lagrangian, and }\alpha\in i\R,\\
 q_{\alpha,W}(\lambda),\ \text{where }W \text{is any Lagrangian,} \ \text{and }\alpha\in\C\setminus(\R\cup i\R)
 \end{gather*}
 generate the twisted loop group $\rlgmtw(\CSp(n,\C),\C^{2n})$.
 \end{thm}
\begin{proof}
Let $g\in \rlgmtw(\CSp(n,\C),\C^{2n})$. We deal with the purely imaginary and non-imaginary poles of $g$ separately, as in the generating theorem for $\Un(n)/\Or(n)$ \cite{Terng2006}. First assume that $\alpha$ is a non-imaginary pole.  By Theorem \ref{thm:gencsp}, we can find a Lagrangian $W$ such that $p_{\alpha, W}g$ has a pole of lower degree than $g$. Then, $q_{\alpha,W}g$ does so, too, and also satisfies the twisting condition. In this way, we remove the pole at $\alpha$, and the possibly occurring zero can then be dealt with in the same manner.  In this way we can remove all of the nonimaginary poles.

Assume that $\alpha=ir$ is a purely imaginary pole of $g$, and write its power series expansion around $\al$ as 
\[ g(\la)=\sum_{j=-k}^\infty \left(\frac{\la-ir}{\la+ir}\right)^j g_j
\]
 with $g_{-k}\neq 0$. We want to show that $V=\Im g_{-k}$ is contained in a real Lagrangian subspace. As before, because $g$ is a loop in $\CSp(n,\C)$, the image of $g_{-k}$ is either equal to $\C^{2n}$ or $\omega$-isotropic. We get additional information by combining the reality and twisting conditions.
They read
\[
g(\la)=\hd g(\bar{\la})^{-1} \text{ and } g(\la)=-Jg(-\la)J,
\]
and imply that
\[
(g(-\bar{\la})v,Jg(\la)Jw)=-(v,w).
\]
Expanding this as a power series, the lowest order term gives us
\[
(g_{-k}v,Jg_{-k}Jw)=0
\]
for all $v$ and $w$, i.e.~$V\perp J V$. In particular, $V$ cannot be equal to $\C^{2n}$, so it is also $\omega$-isotropic. Therefore, $\omega(\bar{V},V)=(V,JV)=0$, i.e.~$V$ is contained in the real isotropic space $V+\bar{V}$, which in turn is contained in some real Lagrangian $W$. We can therefore, proceed as in Theorem \ref{thm:gencsp} using simple factors corresponding to such $W$: first we remove the pole, and then any possible resulting zero.
\end{proof} 

\section{$\SO(n,\C)$}

\subsection{Generating $\rlgm(\SO(n,\C),\C^n)$} \label{sec:genso}

In this section, we prove that for the standard representation of $\SO(n,\C)$, the simple elements defined in Section \ref{sec:simpleelements} generate the rational loop group $\rlgm(\SO(n,\C),\C^n)$. The compact real form $\SO(n)$ is given by the involution $\tau(A)=\ol{A}$.
Let $\ft\subset \so(n)$ be a maximal abelian subalgebra with complexification $\h\subset \so(n,\C)$, and let $\pm \mu_1,\ldots,\pm \mu_r$ be the nonzero weights of the standard representation, where $r$ is the rank of $\SO(n)$. Note that $\mu_1,\ldots,\mu_r$ is a basis of $\h^*$ and that zero is a weight if and only if $n$ is odd.  Let $L_i$ be the weight space corresponding to $\mu_i$; then, $\overline{L_i}$ is the weight space of $-\mu_i$. Both $L_i$ and $\overline{L_i}$ are isotropic lines in $\C^n$. The weight space decomposition is
\[
\C^n=V_0\oplus \bigoplus_{i=1}^r (L_i\oplus\overline{L_i});
\]
here $V_0$ is the weight space corresponding to the weight $0$, if $n$ is odd, and empty otherwise.
Let $H_i$ be such that $\mu_j(H_i)=\delta_{ij}$. The $H_i$ form a basis of the $\rho$-integral lattice, and by Proposition \ref{prop:simple}, we have
\[
\exp\left(\ln\left(\frac{\lambda-\alpha}{\lambda-\bar\alpha}\right)H_i\right)=\left(\frac{\lambda-\bar\alpha}{\lambda-\alpha}\right)\pi_{\overline{L_i}}+\pi_{(L_i\oplus \overline{L_i})^\perp}+\left(\frac{\lambda-\alpha}{\lambda-\bar\alpha}\right)\pi_{L_i}.
\]
Conjugating $H_i$ with elements of $\SO(n)$, we find that the loops 
\be\label{eq:simpleso}
p_{\al,L}=\left(\frac{\lambda-\bar\alpha}{\lambda-\alpha}\right)\pi_{\ol{L}} +\pi_{(L\oplus\bar L)^\perp} +\left(\frac{\lambda-\alpha}{\lambda-\bar\alpha}\right)\pi_{L},
\ee
where $L$ is an arbitrary isotropic line in $\C^n$, are simple elements in $\rlgm(\SO(n,\C),\C)$. 

\begin{thm}\label{thm:genso}
The simple elements $p_{\alpha,L}$, where $\alpha\in\C\setminus \R$ and $L\subset\C^n$ is an isotropic line, generate $\rlgm(\SO(n,\C),\C^n)$.
\end{thm}



\begin{proof}
Let $g\in \rlgm(\SO(n,\C),\C^{n})$ and fix a pole $\al\in \C\setminus \R$. As before, we use induction on its degree to remove it.  The reality condition implies that $\alpha$ is a pole if and only if $\bar{\al}$ is a pole.  Thus, in contrast to the case of $\CSp(n,\C)$, removing a pole at $\al$ simultaneously removes the pole at $\bar{\al}$.  

We write the Laurant expansion of $g$ in  $\frac{\la-\al}{\la-\bar{\alpha}}$ explicitly as
\be
g(\la)=\sum_{j=-k}^\infty \left(\frac{\la-\al}{\la-\ol{\alpha}}\right)^j g_j
\ee
 with $g_{-k}\neq 0$. 
Since $g$ takes values in $\SO(n,\C)$, we have $\langle g(\la) v,g(\la)w\rangle=\langle v,w\rangle $ for all $v,w\in \C^n$ and all $\lambda$ near 0, where $\langle \cdot,\cdot \rangle$ is the symmetric bilinear form on $\C^n$ extending the standard inner product on $\R^n$ complex linearly. Expanding this expression in $\lambda$, the terms of lowest order read
\begin{align}
\langle g_{-k}v,g_{-k}w\rangle&=0\label{eq:gorthcond1}\\
\langle g_{-k+1}v,g_{-k}w\rangle +\langle g_{-k}v,g_{-k+1}w\rangle&=0.\label{eq:gorthcond2}
\end{align}
The first of these equations says that $\Im g_{-k}$ is isotropic. Since no real subspaces are isotropic, we obtain a decomposition
\[\C^n=\Im g_{-k}\oplus (\Im g_{-k}\oplus\ol{\Im g_{-k}})^\perp\oplus\ol{\Im g_{-k}}.\]
Now let $L\subset\Im g_{-k}$ be a line and calculate
\begin{align*}
\sum_{j=-k-1}^\infty &\left(\frac{\la-\al}{\la-\ol{\alpha}}\right)^j \tilde g_j:=\tilde{g}(\lambda):=p_{\al,L}(\lambda)g(\lambda)\\
&=\left(\frac{\la-\al}{\la-\ol{\alpha}}\right)^{-k-1}\pi_{\bar L}g_{-k} 
 +\left(\frac{\la-\al}{\la-\ol{\alpha}}\right)^{-k}(\pi_{\bar L}g_{-k+1}+\pi_{(L\oplus\bar L)^\perp}g_{-k}) +\ldots
\end{align*}
Since $\Im g_{-k}$ is isotropic we have
\[\ker\pi_{\bar L}=L^\perp\supset\Im g_{-k}^\perp\supset\Im g_{-k},\]
 hence $\tilde{g}_{-k-1}$ vanishes.
 
It remains to show that the rank $\tilde g_{-k}$ is smaller than the rank of $g_{-k}$. For this we compute kernels. Let $v\in\ker g_{-k}$. Then, by \eqref{eq:gorthcond2}, we have $\langle g_{-k+1}v,g_{-k}w\rangle=0$  for all $w\in\C^n$. Thus $g_{-k+1}(\ker g_{-k})\perp\Im g_{-k}\supset L$. In particular we have $\ker g_{-k}\subset \ker\tilde g_{-k}$.
To see that this is a proper inclusion we observe that there exists some $v\in\C^n$ such that $g_{-k}v\in L\setminus\{0\}$. Applying \eqref{eq:gorthcond2} with $w=v$ yields
\[\langle g_{-k+1}v,g_{-k}v\rangle=0.\]
Thus $g_{-k+1}v\perp L$ and so $v\in\ker\tilde g_{-k}$. We have thus found a vector 
\[
v \in \ker\tilde g_{-k}\setminus\ker g_{-k}.
\]
Conseqently $\rank \tilde g_{-k}<\rank g_{-k}$ and $\tilde g$ has a pole of lower degree at $\al$.
\end{proof}

\subsection{Dressing and permutability} 
The following theorem is well-known, see e.g.\ \cite{Burstall2004}. 
\begin{thm} \label{thm:dressingso} Let $h\in \mcl(\SO(n,\C))$ be holomorphic at $\al\in \C\setminus \R$. If we define $L'=h(\alpha)^{-1}L$,
then
\be \label{eq:dressingso}
p_{\al,L}hp_{\al,L'}^{-1}
\ee
is holomorphic at $\al$.
\end{thm}
\begin{proof} Let $h(\la)=\sum_{i} \left(\frac{\la-\al}{\la-\ol{\al}}\right)^i h_i$.
Multiplying out \eqref{eq:dressingso}, the $\left(\frac{\la-\al}{\la-\ol{\al}}\right)^{-2}$- and $\left(\frac{\la-\al}{\la-\ol{\al}}\right)^{-1}$-coefficients are
\begin{align}
&\pi_{\ol{L}}h_0\pi_{L'} \label{eq:dressingcoeff1}\\
&\pi_{\ol{L}}h_0\pi_{(L'\oplus\ol{L'})^\perp}+ \pi_{(L\oplus\ol{L})^\perp}h_0\pi_{L'} + \pi_{\ol{L}}h_1\pi_{L'}.\label{eq:dressingcoeff2}
\end{align}
Since the image of $h_0\pi_L'$ is $L$, the term \eqref{eq:dressingcoeff1} and the second summand of \eqref{eq:dressingcoeff2} vanish. The first summand of \eqref{eq:dressingcoeff2} is zero because of
\[
\left<h_0(L'\oplus\ol{L'})^\perp,L\right>=\left<h_0(L'\oplus\ol{L'})^\perp,h_0L'\right>=\left<(L'\oplus\ol{L'})^\perp,L'\right>=0.
\]
For the vanishing of the third note that $h_1h_0^{-1}\in \so(n,\C)$. Then,
\[
\left<h_1L',L\right>=\left<(h_1h_0^{-1})L,L\right>=0.
\]
It now is possible to evaluate the loop \eqref{eq:dressingso} at $\al$. The determinant of the loop is $1$ at $\al$, so since $\SO(n,\C)$ is closed in $\GL(n,\C)$, the loop is holomorphic at $\al$.
\end{proof}

\begin{cor}
Suppose that $p_{\alpha,L}$ and $p_{\beta,M}$ are simple factors such that $\al\neq \beta,\bar{\beta}$. Then
\[p_{\beta,p_{\alpha,L}(\beta)M}p_{\alpha,L}= p_{\alpha,p_{\beta,M}(\alpha)L}p_{\beta,M}.\]
\end{cor}

\subsection{Generators for the twisted loop groups}

Let $\sigma$ be a holomorphic involution of $\SO(n,\C)$ satisfying $\tau \sigma=\sigma \tau$, thereby giving rise to a symmetric space $\SO(n)/K$, where $K$ is the fixed point set of $\sigma$ in $\SO(n)$. Denote the derivative $\ed\sigma_1\in\Aut(\g)$ by $\sigma$ also. Recall that we say a loop $g$ is {\bf twisted} if
\[\sigma(g(-\lambda))=g(\lambda).\]
Note that for any such $\sigma$, there exists $s\in \Or(n)$ such that $\sigma(A)=sAs^{-1}$. In the case of the Grassmannians, $s$ is given by 
\[
s=\left(\begin{matrix}I_k & 0 \\ 0 & -I_{n-k}\end{matrix}\right)
\]
for some $k$ and for $\SO(2m)/\Un(m)$,
\[
s=\left(\begin{matrix}0 & I_m \\ -I_m & 0 \end{matrix}\right).
\]
Note in particular that $s\in \Or(n)$ sends isotropic lines in $\C^n$ onto isotropic lines.


\begin{lem}\label{lem:perm}
If $\alpha\notin i\R$, and $p_{\alpha,L}$ is a simple element, then the product
\[q_{\alpha,L}(\lambda):=
p_{-\alpha,p_{\alpha,L}(-\alpha)s L}p_{\alpha,L}= p_{\alpha,p_{-\alpha,s L}(\alpha)L}p_{-\alpha,sL}\]
satisfies the twisting condition.
\end{lem}

The proof follows immediately from the permutability proposition.

\begin{thm}
The elements
\begin{gather*}
p_{\alpha,L}(\lambda)=\left(\frac{\lambda-\bar\alpha}{\lambda-\alpha}\right)\pi_{\ol{L}} +\pi_{(L\oplus\bar L)^\perp} +\left(\frac{\lambda-\alpha}{\lambda-\bar\alpha}\right)\pi_{L},\quad\text{where }sL=\bar L,\ \text{and }\alpha\in i\R,\\
{\text and}\;\;\; q_{\alpha,L}(\lambda),\quad\text{where }\alpha\in\C\setminus(\R\cup i\R),\ \text{and }L\subset\C^n\ \text{isotropic,}
\end{gather*}
generate the twisted loop group $\rlgmtw(\SO(n,\C),\C^n)$.
\end{thm}

\begin{proof} Let $g\in \rlgmtw(\SO(n,\C),\C^n)$. We proceed by induction on degree, considering the cases of purely imaginary poles and non-imaginary poles separately.  First assume that $\alpha$ is a non-imaginary pole of $g$. Observe that then,  $\bar\alpha, -\alpha,$ and $-\bar\alpha$ are also poles of $g$. If we manage to remove the pole at $\alpha$ by successively multiplying with loops satisfying the reality and twisting conditions, we automatically remove these other poles as well. 
Applying Theorem \ref{thm:genso}, we see that there exists an isotropic line $L$ such that
$p_{\alpha,L}g$
has lower total degree at $\alpha$. Then, $q_{\alpha,L}g$ does so, too, and also satisfies the twisting condition. Continuing in the same way, we are able to remove the pole at $\alpha$.

Suppose that $\alpha=ir$ is a purely imaginary pole of $g\in \rlgmtw(\SO(n,\C),\C^n)$ of degree $k$, and expand $g$ as before. 
The twisting and reality conditions
\[
\bar{g(\bar{\la})}=g(\la)=\sigma g(-\la)
\]
 imply 
\[
\sum_{j=-k}^\infty \left(\frac{\la+ir}{\la-ir}\right)^j \bar{g_j}  =\sum_{j=-k}^\infty \left(\frac{\la+ir}{\la-ir}\right)^j \sigma { g_j},\]
hence $\sigma g_j=\ol{g_j}$ for all $j$. In particular $\sigma g_{-k}=\ol{ g_{-k}}$, from which we see that $\ol{\Im g_{-k}}=\Im{\ol{g_{-k}}}=\Im\sigma g_{-k}=s(\Im g_{-k})$. Therefore, $s$ composed with conjugation is an involutive endomorphism of $\Im g_{-k}$ and we may choose a fixed line $L$: i.e., $s\bar L=L\subset\Im g_{-k}$.

By the proof of Theorem \ref{thm:genso}, the map $p_{ir,L}g$ has a pole of lower total degree at $ir$ than $g$. This loop clearly satisfies the reality and twisting conditions and is thus in $\rlgmtw(\SO(n,\C),\C^n)$. By induction we may entirely factor out the poles on $i\R^\times$.

As before we conclude from Liouville's theorem that we have reduced the given loop to the identity using simple factors.
\end{proof}

\section{$\GGC$}

\subsection{The linear algebra of $\GGC$ and its fundamental representation}\label{sec:g2basics}
Let $\Oc$ denote the octonions, the unique real $8$-dimensional division algebra, equipped with the natural metric $\langle x,y \rangle=\Re(x\cdot \bar{y})=\frac{1}{2}(x\cdot \bar{y}+y \cdot \ol{x})$. The compact simple Lie group $\GG$ is known to be the automorphism group of $\Oc$. Since the metric is defined via the multiplication, we get $\GG\subset \SO(\Oc)$. The subspace $\R\cdot 1\subset \Oc$ is fixed by $\GG$, so if we identify $\Im(\Oc)=\R^7$, we obtain the fundamental representation $\GG\subset \SO(7)$.

Let $1,e_1,\ldots,e_7$ be the standard orthonormal basis of $\R^8 \cong \Oc$.  The octonionic multiplication table is displayed in Appendix \ref{sec:g2}.

\begin{defn} \hspace{1cm}\\ \vspace{-0.4cm}
\begin{itemize}
\item A $3$-plane $A \subset \R^7$ is {\bf associative} if it generates an associative subalgebra of $\Oc$.
\item A $4$-plane $C \subset \R^7$ is {\bf coassociative} if $C^\perp$ is associative.
\end{itemize}
\end{defn}
\begin{rem}
In the literature the title \us{(}co\us{)}associative also implies the choice of a particular orientation which makes the plane calibrated \cite{Harvey1982}.  We do not need the calibrations or the orientations and so use a slightly weaker definition.
\end{rem}
Our convention was chosen so that $e_1 \w e_2 \w e_3 \w e_4$ is coassociative. As they will be heavily relied upon here, we record the following well known facts concerning the action of $\GG$ on $\R^7$.

\begin{lem} \label{lem:realg2trans}
\hspace{1cm}\\ \vspace{-0.4cm}
\begin{enumerate}
\item{Each $2$-plane $E \subset \R^7$ is contained in a unique associative $3$-plane $E_+$.}
\item{$\GG$ acts transitively on the Grassmannian of oriented $2$-planes $E \subset \R^7$ with stabilizer $\Un(2)$ acting irreducibly on $(E_+)^{\perp}$.     }
\end{enumerate}
\end{lem}

In $\SO(4)=\stab(e_1 \w e_2 \w e_3 \w e_4)$ we can choose the torus $T$ whose Lie algebra is the span of \\
\centerline{
$H_1=\bp 
0&-1&0&0&0&0&0\\
1&0&0&0&0&0&0\\
0&0&0&1&0&0&0\\
0&0&-1&0&0&0&0\\
0&0&0&0&0&0&0\\
0&0&0&0&0&0&0\\
0&0&0&0&0&0&0\\
\ep\;\;{\rm and }\;\; H_2=\bp 
0&-1&0&0&0&0&0\\
1&0&0&0&0&0&0\\
0&0&0&-1&0&0&0\\
0&0&1&0&0&0&0\\
0&0&0&0&0&0&0\\
0&0&0&0&0&0&-2\\
0&0&0&0&0&2&0\\
\ep.$}

Let $\Oc_\C=\Oc\otimes \C$; then, $\GGC:=\GG^\C\subset \SO(7,\C)$ is the automorphism group of $\Oc_\C$.  From now on, if $x,y\in \C^7=\Im(\Oc)\otimes\C$, the product $x\cdot y\in \C^7$ will denote the octonionic-imaginary part of the product of $x$ and $y$.  The torus $T$ induces the following weight diagram of $\C^7=\Im(\Oc)\otimes\C$:

\begin{center}
\begin{figure}
\includegraphics[scale=.8]{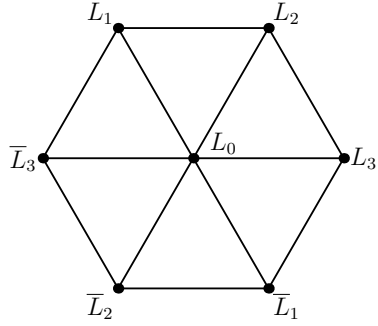}
\caption{Weight diagram for the fundamental representation of $\GGC$.}
\end{figure}
\end{center}

The weights are $\mu_0=0$, $\pm \mu_1,\pm \mu_2,\pm \mu_3$, where
\begin{align*}
\mu_1(H_1)=-i,\;\; \mu_1(H_2)=-i\\
\mu_2(H_1)=-i, \;\; \mu_2(H_2)=i\\
\mu_3(H_1)=0, \;\; \mu_3(H_2)=2i
\end{align*}
and the weight spaces $L_i$ of $\mu_i$ are given by  $L_1=\C \cdot \{e_1+ie_2\}$, $L_2=\C \cdot \{e_3-ie_4\}$, $L_3=\C \cdot \{e_6-ie_7\}$, and  $L_0=\C \cdot \{e_5\}$.  The product $L_i\cdot L_j$ is the weight space for $\mu_i+\mu_j$, if this is a weight, and zero otherwise.  For example, $L_1\cdot L_3=L_2$. 

\begin{defn}
An isotropic $2$-plane $\mcc \subset \C^7$ satisfying $\mcc\cdot\mcc=0$ is {\bf complex coassociative}.  
\end{defn}
Either from the weight diagram or the multiplication table in the appendix, one calculates that $L_1\cdot L_2=0$, i.e.~$L_1\oplus L_2$ is complex coassociative.
Any complex coassociative plane $\mcc$ is isotropic by definition and thus induces an orthogonal decomposition 
\begin{equation}
\C^7=\mcc\oplus \mca\oplus \ol{\mcc},\label{eq:decomp}
\end{equation}
where $\mca=(\mcc\oplus\ol{\mcc})^\perp$ is called {\bf complex associative}.

\begin{lem}\label{lem:setofc}
Let $L \subset \C^7$ be an isotropic line.  
\begin{enumerate}
\item
There is a unique isotropic $2$-plane $\mcb \subset (L \oplus \bar{L})^{\perp}$ such that any line in $\mcb$ multiplies with $L$ to be zero.  Conversely, for any complex coassociative $\mcc \supset L$ we can choose a $K \in \mcb$ such that $\mcc=L \oplus K$.
\item
If $K$ is any other line, then $L \cdot K=0$ implies that $K$ is isotropic and therefore $L \oplus K$ is complex coassociative. 
\end{enumerate}
\end{lem}
\begin{proof}
By part two of Lemma \ref{lem:realg2trans}, we can assume without loss of generality that $L=L_1$.  Then $(L\oplus \ol{L})^\perp=L_2\oplus \ol{L_2}\oplus L_3\oplus \ol{L_3}\oplus L_0$.  A look at the weight diagram shows that $\mcb=L_2\oplus \ol{L_3}$, which shows part one.  Part two follows from the fact that $K$ has to lie in the isotropic $3$-plane $L\oplus \mcb$.
\end{proof}

\begin{prop}\label{prop:coasstrans}
The compact real group $\GG$ acts transitively on the space of complex coassociative planes.
\end{prop}
\begin{proof}
Let $\mcc=L\oplus K$ be a decomposition of a complex coassociative plane into lines such that $L\perp \ol{K}$. 
By part two of Lemma \ref{lem:realg2trans}, there is $g\in \GG$ with $gL=L_1$. By part one of Lemma \ref{lem:setofc}, we know that $gK\subset L_2\oplus \ol{L_3}$. Then by part two of Lemma \ref{lem:realg2trans}, there exists $h\in \GG$ with $hL_1=L_1$ and $hgK=L_2$. We have shown that any complex coassociative plane can be mapped to $L_1\oplus L_2$.
\end{proof}

\begin{cor}
The intersection $C=(\mcc\oplus \ol{\mcc})\cap \R^7$ is a coassociative $4$-plane and $\mca$ is the complexification of an associative $3$-plane $A \subset \R^7$.
\end{cor}

\begin{cor}\label{cor:orth}
If $\mcc$ and $\mcd$ are complex coassociative and $\mcc \cap \mcd \neq 0$, then they are orthogonal.
\end{cor}

\begin{lem} 
The octonionic multiplication table for the decomposition \eqref{eq:decomp} is given by
\begin{center}
\begin{tabular}{c||c|c|c}
 & $\mcc$ & $\mca$ & $\ol{\mcc}$ \\ \hline\hline
$\mcc$ & $0$ & $\mcc$ & $\mca$ \\ \hline
$\mca$ & $\mcc$ & $\mca$ & $\ol{\mcc}$ \\ \hline
$\ol{\mcc}$ & $\mca$ & $\ol{\mcc}$ & $0$
\end{tabular}
\end{center}
\end{lem}
\begin{proof}
Because $\GG$ acts transitively on complex coassociative planes, one only has to check this using the multiplication table from Appendix \ref{sec:g2} for the standard complex coassociative plane $L_1\oplus L_2$.
\end{proof}

\subsection{Generating $\rlgm(\GGC,\C^7)$} \label{sec:geng2}
As before, let $\GGC\subset \SO(7,\C)$ be the complexification of $\GG\subset \SO(7)$; conjugation on $\SO(7,\C)$ restricts to an involution $\tau$ on $\GGC$ whose fixed point set is $\GG$. First we need to describe the simple elements.

We use the notation of Section \ref{sec:g2basics}.  The element $iH_1$ is $\rho$-integral.  By Proposition \ref{prop:simple}, the simple element $\exp\left(\ln\left(\frac{\la-{\al}}{\la-\ol{\al}}\right)iH_1\right)$ acts on $\C^7$ as
\[
\left(\frac{\la-\ol{\al}}{\la-\al}\right)\pi_{\ol{\mcc_0}}+\pi_{\mca_0}+\left(\frac{\la-\al}{\la-\ol{\al}}\right) \pi_{\mcc_0},
\]
where $\mcc_0=L_1\oplus L_2$ and $\ol{\mcc_0}=\ol{L_1}\oplus \ol{L_2}$ are complex coassociative planes and ${\mca_0}=L_3\oplus L_0\oplus \ol{L_3}$ is a complex associative $3$-plane.  Conversely, given any element of the form
\[
p_{\al,\mcc}(\la):=\left(\frac{\la-\ol{\al}}{\la-{\al}}\right)\pi_{\ol{\mcc}}+\pi_{\mca}+\left(\frac{\la-{\al}}{\la-\ol{\al}}\right) \pi_{\mcc},
\]
where $\mcc$ is an arbitrary complex coassociative plane and $\mca=(\mcc\oplus \bar\mcc)^\perp$, Proposition \ref{prop:coasstrans} implies that it is conjugate to the simple element defined by $H_1$, so it is a simple element itself.

\begin{thm} The simple elements $p_{\alpha,\mcc}$ generate $\rlgm(\GGC,\C^7)$.
\end{thm}

\begin{proof} Let $\alpha$ be a pole of $g$ and expand $g$ as before. Because $g$ is a map into $\SO(7,\C)$, Equations \eqref{eq:gorthcond1} and \eqref{eq:gorthcond2} may be used here:
\begin{align}
\langle g_{-k} (v),g_{-k}(w)\rangle&=0 \label{eqn:so1}\\
\langle g_{-k} (v),g_{-k+1}(w)\rangle +\langle g_{-k+1} (v),g_{-k}(w)\rangle&=0.\label{eqn:so2}
\end{align} 
For all $v,w\in \C^7$, we have $g(\lambda)(v\cdot w)=g(\lambda)(v)\cdot g(\la)(w)$; expanding this equation and comparing the two coefficients of lowest order, we obtain
\begin{align}
 g_{-k} (v)\cdot g_{-k}(w)&=0\label{eqn:g21}\\
 g_{-k} (v)\cdot g_{-k+1}(w) + g_{-k+1} (v)\cdot g_{-k}(w) &=\begin{cases}0 & k\ge 2 \\ g_{-1}(v \cdot w ) & k=1\end{cases}.\label{eqn:g22}
\end{align} 

It follows from Equations \eqref{eqn:so1} and \eqref{eqn:g21} that $\Im(g_{-k})$ is either an isotropic line or a complex coassociative plane, thus  $\rank(g_{-k}) \leq 2$.  

Assume $k \geq 2$; the case of $k=1$ will be dealt with later. Choose $L \subset \Im(g_{-k})$ and $v$ such that $0 \neq g_{-k}(v) \in L$.  By Lemma \ref{lem:setofc} there exists a line $M \subset (L \oplus \bar{L})^{\perp}$ such that $\mcc=L \oplus M$ is complex coassociative and $g_{-k+1}(v) \perp M$. Note that $\mcc$ is perpendicular to $\Im(g_{-k})$, since $\Im(g_{-k})$ is either equal to $L\subset \mcc$ or a complex coassociative plane intersecting $\mcc$ nontrivially, whereupon Corollary \ref{cor:orth} is applicable.

We claim that $p_{\al,\mcc} g$ has a pole of strictly lower total degree at zero. Multiplying out, we get
\begin{align*}
p_{\al,\mcc}(\la) g(\lambda)&=\left(\frac{\la-\al}{\la-\ol{\al}}\right)^{-k-1}\pi_{\mccb}g_{-k}+\left(\frac{\la-\al}{\la-\ol{\al}}\right)^{-k}(\pi_{\mccb}g_{-k+1}+\pi_{\mca} g_{-k})+\ldots\\
&=:\sum_{j=-k-1}^\infty \left(\frac{\la-\al}{\la-\ol{\al}}\right)^j \tilde g_j.
\end{align*}
Since $\mcc$ and $\Im(g_{-k})$ are orthogonal, $\tilde g_{-k-1}=\pi_{\mccb}g_{-k}=0$. We need to show that the rank of $\tilde g_{-k}$ is smaller than the rank of $g_{-k}$. For this, we compare the kernels. 

If $w \in \ker(g_{-k})$ then Equation \eqref{eqn:g22} gives 
\be
g_{-k} (v)\cdot g_{-k+1}(w) =0
\ee
which implies that $g_{-k+1}(w)\cdot L=0$, i.e.~$g_{-k+1}(w)$ is either contained in $L$, which immediately implies that it is perpendicular to $\mcc$, or, together with $L$ spans a complex coassociative plane, whereupon it is perpendicular to $\mcc$ by Corollary \ref{cor:orth}. It follows that $\tilde g_{-k}(w)=0$, i.e.~$\ker g_{-k}\subset \ker \tilde g_{-k}$. 

To show that we have a strict inclusion, we show that $\tilde g_{-k}(v)=0$. We have $\pi_{\mca} g_{-k}(v)=0$, so we need to show that $g_{-k+1}(v)$ is perpendicular to $\mcc$.
By definition, $g_{-k+1}(v)$ is perpendicular to $M$.  Equation \eqref{eqn:so2} with $v=w$ implies that it is perpendicular to $L$, so it is perpendicular to $\mcc$.
 

Now assume $k=1$.  From Equation \eqref{eqn:g22} we find that $\ker(g_{-1})$ is a multiplicatively closed subspace of $\C^7$. Since such subspaces can at most be five-dimensional, it follows that $\rank(g_{-1})=2$. The loop $g$ may also be regarded as a loop in $\SO(7,\C)$, so applying the Generating Theorem \ref{thm:genso} for $\rlgm(\SO(7,\C),\C^7)$, we see that there exist two isotropic lines $L,K\subset \C^7$ such that
\be \label{eqn:decompk=1}
g=p_{\al,L}p_{\al,K}\tilde g,
\ee
where $\tilde{g}$ is holomorphic at $\al$, and $p_L$ and $p_K$ are given by \eqref{eq:simpleso}. Multiplying out \eqref{eqn:decompk=1}, the coefficient of order $-2$ is $\pi_{\bar L} \pi_{\bar K} \tilde g(\al)$. Because $\tilde{g}$ has a pole of order $1$ at $\al$ ($k=1$) and $\tilde g(\al)$ is invertible, it follows that $K\subset {\bar L} ^\perp$.
The coefficient of order $-1$ is
\be
(\pi_L\pi_{(K\oplus \bar K)^\perp} + \pi_{(L\oplus \bar L)^\perp} \pi_K) \tilde g(\al),
\ee
which has the two-dimensional image $\mcc:=L\oplus \pi_{(L\oplus \bar L)^\perp} K$. The plane $\mcc$ is complex coassociative by \eqref{eqn:g21}. By Proposition \ref{prop:coasstrans}, without loss of generality we can assume that $L=L_1$ and $\pi_{(L\oplus \bar L)^\perp} K=L_2$ in the weight diagram. We have that $K\subset \pi_{(L\oplus \bar L)^\perp} K \oplus \bar L=L_2\oplus \ol{L_1}$. Note that $p_{\al,L} p_{\al,K}$ is not necessarily equal to $p_{\al,\mcc}$; this is only the case if $K$ is contained in $(L\oplus \bar{L})^\perp$, i.e.~$K=L_2$. Having a look at the weight diagram, we see that the isotropic line $R=L_3\subset \mca$ satisfies $L\cdot \ol{R}=0$ and $K\cdot R =0$ so that both $p_{\al,L}p_{\al,\ol{R}}$ and $p_{\al,R} p_{\al,K}$ are simple factors. Thus
\be
g=p_{\al,L}p_{\al,K}\tilde g=(p_{\al,L} p_{\al,\ol{R}}) (p_{\al,R} p_{\al,K}) \tilde g.
\ee
has been factored.
\end{proof}

The proof of the $k=1$ case of the theorem implies the following proposition, which will be used in the discussion of the dressing transformation and permutability.

\begin{prop}\label{prop:simplepolesg2}
A pair of isotropic lines $L,K$ satisfies
\be\label{eq:pairoflines}
\begin{split}
K &\subset {\ol{L}}^{\perp} {\text{ and}}\\ 
\mcc&=L \oplus \pi_{(L\oplus \ol{L})^{\perp}}K \text{ is complex coassociative}
\end{split}
\ee
if and only $p_{\al,L}p_{\al,K} \in \rlgm(\GGC,\C^7)$.  In this case the product only has simple poles at $\al$ and $\ol{\al}$.   In fact, any $g \in \rlgm(\GGC,\C^7)$  whose only poles are simple ones at $\al$ and $\ol{\al}$ with $\al \in \C \setminus \R$ can be factored as $g=p_{\al,L}p_{\al,K}$.  Furthermore, either
\begin{enumerate}
\item $K \subset (L\oplus \ol{L})^{\perp}$, in which case $\mcc=L \oplus K$ and 
\[ g=p_{\al,\mcc}\]  or
\item $K \not \subset (L\oplus \ol{L})^{\perp}$ and then there exists an isotropic line $R$ such that $\mcc_1=L \oplus R$ and $\mcc_2=\ol{R} \oplus  K$ are complex coassociative and
\[
g=p_{\alpha,\mcc_1}p_{\al,\mcc_2}.
\]
\end{enumerate}
\end{prop}

\subsection{Dressing and permutability}
Any rational loop whose only poles are simple ones at $\al$ and $\ol{\al}$ with $\al\in \C\setminus \R$ is, by Proposition \ref{prop:simplepolesg2}, of the form $p_{\al,L,K}=p_{\al,L}p_{\al,K}$ for two isotropic lines $L$ and $K$ for which \eqref{eq:pairoflines} holds.
The notation $p_{\al,L,K}$ will only be used for lines $L$ and $K$ satisfying these conditions.   Dressing of positive loops with $p_{\al,L,K}$ is performed by using the dressing action in $\SO(7,\C)$ for each of the $\SO(7,\C)$-factors.  In order to prove a permutability formula we also need to use a dressing-type action of the rational loop group on itself.

 Let $h \in \rlgm(\GGC,\C^7)$ be holomorphic at $\al$ \us{(}and thus $\ol{\al}$\us{)}. Define 
\begin{align}
K'&=h({\al})^{-1}{K}\\
p_{\al,K}*h&=p_{\al,K}hp_{\al,K'}^{-1}.\label{eq:supK}
\end{align}
From the dressing theorem for $\SO(n,\C)$ we know that $p_{\al,K}*h$ is holomorphic at $\al$, allowing us to define
\be
L'=(p_{\al,K}*h)({\al})^{-1}{L}.
\ee
\begin{prop}
The pair of lines $L', K'$ satisfies \eqref{eq:pairoflines}. Therefore, $p_{\alpha,L',K'}\in \rlgm(\GGC,\C^7)$ has simple poles at $\alpha$ and $\ol{\al}$. Furthermore, 
\be
p_{\al,L,K}hp_{\al,L',K'}^{-1}\label{eq:nopole}
\ee
is in $\rlgm(\GGC,\C^7)$ and holomorphic at $\al$ and $\bar{\al}$. 
\end{prop}
\begin{proof}
That the resulting element is holomorphic follows from applying Theorem \ref{thm:dressingso} twice. We will now show that it is a loop in $\GGC$.  By the generating theorem for $\rlgm(\GGC,\C^7)$, 
\[
p_{\al,L,K}h=\hat{h}p_{\al,M,N}
\]
where $\hat{h}$ is holomorphic at $\al$ and $\ol{\al}$ and $p_{\al,M,N}$ is a loop in $\GGC$.  That we do not have a more complicated product of simple elements having poles at $\al$ and $\ol{\al}$ is due to the fact that the left hand side has only simple poles there.  Therefore 
\[
p_{\al,L,K}hp_{\al,L',K'}^{-1}\hat{h}^{-1}=\hat{h}p_{\al,M,N}p_{\al,L',K'}^{-1}\hat{h}^{-1}.
\]
The left hand side is holomorphic at $\al$, so $p_{\al,M,N}p_{\al,L',K'}^{-1}$ is constant. Since all factors were normalized to be the identity at infinity, $p_{\al,L',K'}=p_{\al,M,N}$, and is thus in $\rlgm(\GGC,\C^7)$. 
\end{proof}

\begin{cor}[Permutability]\label{cor:g2perm}
Let $L,K$ and $M,N$ be pairs of lines satisfying \eqref{eq:pairoflines}. Then, for all $\al ,\beta \in \C \setminus \R$ with $\alpha \neq \beta, \ol{\beta}$, we have 
\[
p_{\al,L',K'}p_{\beta,M,N}=p_{\beta,M',N'}p_{\al,L,K}
\]
where
\begin{align*}
K'&=p_{\beta,M,N}({\al}){K}\\
L'&=(p_{\al,K'}*p_{\beta,M,N})({\al}){L}\\
N'&=p_{\al,L,K}(\beta){N}\\
M'&=(p_{\beta,N'}*p_{\al,L,K})(\beta){M}.
\end{align*}
\end{cor}

\subsection{Generators for the twisted loop group}\label{sec:twistedg2}
There is a unique irreducible compact Riemannian symmetric space whose isometry group is $\GG$.  It is defined by the automorphism $\sigma(g)=sgs^{-1}$, where \[s=\bp -I_4 & 0 \\ 0 & I_3 \ep\in \GG,\]
and $\GG^\sigma=\SO(4)$.  In this section we introduce generators for the twisted loop group $\mcl_{-}^{\tau,\sigma}(\GGC,\C^7)$.

\begin{lem} For $p_{\al,L,K}$ as above, the loop 
\[
q_{\al,L,K}=p_{-\al,(sL)',(sK)'}p_{\al,L,K}
\] satisfies the twisting condition $\sigma(q_{\al,L,K}(-\la))=q_{\al,L,K}(\la)$, where
\begin{align*}
(sK)'&=p_{\al,L,K}(-\al)sK\\
(sL)'&=(p_{-\al,(sK)'}*p_{\al,L,K})(-\al)sL.
\end{align*}
\end{lem}
\begin{proof} Corollary \ref{cor:g2perm} yields 
\be\label{eq:g2perm}
q_{\al,L,K}=p_{-\al,(sL)',(sK)'}p_{\al,L,K}=p_{\al,L',K'}p_{-\al,sL,sK},
\ee
where 
\begin{align*}
K'&=p_{-\al,sL,sK}(\al)K\\
L'&=(p_{\al,K'}*p_{-\al,sL,sK}(\al)L.
\end{align*}
Note that 
\[
s(sK)'=K'\text{ and } s(sL)'=L'.
\]
The equations
\[
\sigma(p_{\al,L,K}(-\la))=p_{-\al,sL,sK}(\la)
\]
and
\[
\sigma(p_{-\al,(sL)',(sK)'}(-\la))=p_{\al,s(sL)',s(sK)'}(\la)=p_{\al,L',K'}(\la)
\]
combine to
\[
\sigma(q_{\al,L,K}(-\la))=(p_{\al,L',K'}p_{-\al,sL,sK})(\la).
\]
This, together with \eqref{eq:g2perm}, gives the desired result.
\end{proof}

\begin{thm} The elements $p_{\beta,M,N}$, with $\beta \in i\R$ and $sM=\ol{M}$ and $sN=\ol{N}$, and $q_{\al,L,K}=p_{-\al,(sL)',(sK)'}p_{\al,L,K}$ generate $\rlgmtw(\GGC,\C^7)$.
\end{thm}
\begin{rem}
Note that $L$ and $K$ are only assumed to satisfy condition \eqref{eq:pairoflines}.
\end{rem}
\begin{proof} Let $g\in \rlgmtw(\GGC,\C^7)$. We use the same induction proof that was used for the generating theorem for $\SO(n)$-symmetric spaces.  If the order of a pole is greater than one, we may reduce its order using elements of the form $p_{\beta,\mcc}$ and $q_{\al,L,K}$ with $L\oplus K$ being complex coassociative. The only new difficulty occurs for poles of order $1$, so suppose we have applied the induction step until only such poles remain.

We first regard the case of a purely imaginary pole $\beta$ of order $1$.  As in the `$k=1$'-part of the proof of the generating theorem for $\rlgm(\GGC,\C^7)$, we find isotropic lines $M$ and $N$ such that
\[
g=p_{\beta,M,N}h,
\]
where $h$ is holomorphic at $\beta$, the line $N$ can be written as $N=R+t\ol{M}$ with $R\subset (M\oplus \ol{M})^\perp$ and $t\in \C$, and  $\mcc=M\oplus R$ is complex coassociative.  Note that there are many choices for $M$ and $N$; later, we will choose a particular line $M$.
The loop $p_{\beta,M,N}$ has poles only at $\beta$ and $\ol{\beta}$, and because $\beta\in i\R$, the loop $\sigma(p_{\beta,M,N}(-\la))$ has the same poles. Then,  
\[
\sigma(p_{\beta,M,N}(-\la))\sigma(h(-\la))=\sigma(g(-\la))=g(\la)=p_{\beta,M,N}(\la)h(\la)
\]
and therefore
\[
p_{\beta,M,N}(\la)^{-1}\sigma(p_{\beta,M,N}(-\la))=h(\la) \sigma(h(-\la))^{-1}.
\]
The right hand side is holomorphic at $\beta$ and $\ol{\beta}$, so $p_{\beta,L,K}$ satisfies the twisting condition.

The symmetries $\sigma(g)=g$ and $\ol{\beta}=-\beta$ ensure that we may additionally assume $sM=\ol{M}$. Since $p_{\beta,M,N}$ satisfies the twisting, from this choice of $M$ it follows that $sN=\ol{N}$.

Let $\al\in \C\setminus (\R\cup i\R)$ be a pole of $g$ of order $1$.  By the generating theorem for $\rlgm(\GGC,\C^7)$, we find $L$ and $K$ be such that the loop
\[
p_{\al,L,K}g
\]
is holomorphic at $\al$ and $-\al$. Since $p_{-\al,(sL)',(sK)'}$ is holomorphic at $\al$ and $\ol{\al}$, the loop
\[
q_{\al,L,K}g=p_{-\al,(sL)',(sK)'}p_{\al,L,K}g
\]
is also holomorphic at $\al$ and $\ol{\al}$ and satisfies the twisting condition.
\end{proof}

\section{A geometric application:  $\GG/\SO(4)$-abelian surfaces in $\R^7$}

Connections have been made between certain hyperbolic integrable systems associated to symmetric spaces $\G^{\tau}/\K$ (or the corresponding twisted loop group) and submanifold geometries (of symmetric spaces) with flat tangent  bundles \cite{Ferus1996,Terng2005}.   In order to associate a submanifold geometry of $\C^n$ to the $\Un(n)/\Or(n)$-system, Terng and Wang \cite{Terng2006} introduce an affine extension and it this type of extension that will be used here.   As explained more fully in \cite{Fox2008}, the correspondence between the $G^\tau/K$-system and submanifold geometry in the fundamental representation of $G^\tau$ extends to any $\G^{\tau}/\K$ which is a Grassmannian for the fundamental representation of $\G^{\tau}$.  We now explain this in the case of $\GG/\SO(4)$.   Without the affine extension Brander \cite{Brander2007} produced a surface geometry in $\s{6}$. We refer the reader to \cite{Terng2006} for an explanation of how one uses the dressing action to generate new solutions.  One could seek to interpret the dressing transformation on the submanifold geometry as a Ribaucour transformation as was done in \cite{Dajczer2000,Terng2006}.  Our calculations have shown that to calculate an explicit expression for the new immersion into $\R^7$ using the dressing action requires dealing with certain quadratic terms which do not appear for the $\Un(n)/\Or(n)$-system.  We have not pursued these calculations.

The fact that $\GG/\SO(4)$ is a rank two symmetric space suggests that one should seek a geometry of surfaces.  But $\GG/\SO(4)$ is the Grassmannian of associative $3$-planes (and also the Grassmannian of coassociative $4$-planes), which suggests a geometry of $3$-folds (or $4$-folds).   This conflict in dimension can be reconciled using the following
\begin{lem}\label{lem:g2grass}
\hspace{1cm}\\\vspace{-0.4cm}
\begin{enumerate}
\item{Each $2$-plane $E \subset \R^7$ is contained in a unique associative $3$-plane which will be denoted $E_+$.}\label{lem:part1}
\item{$\GG$ acts transitively on the Grassmannian of oriented $2$-planes $E \subset \R^7$ with stabilizer $\Un(2)$.     }
\item{$\GG$ acts transitively on the Grassmannian of associative $3$-planes \us{(}or of coassociative $4$-planes \us{)} and the stabilizer is $\SO(4)$.}
\end{enumerate}
\end{lem}

Let $\phi:M^2 \to \R^7$ be a smooth immersion of a surface.  Let  $\gamma:M \to \gr$ be the associated Gauss map, $\gamma(m)=\phi_*(T_m M) \subset \R^7$.  The fibration \\
\centerline{\xymatrix{ \s{2} \ar[r] & \gr \ar[d]^q \\ & \GG/\SO(4)}}\\
implied by part \ref{lem:part1} of Lemma \ref{lem:g2grass}, where $q(E)=E_+$, leads to the notion of an extended Gauss map:
 \[\Gamma=q \circ \gamma:M \to \GG/\SO(4).
 \]  
The Grassmannian of associative $3$-planes (or equivalently, the Grassmannian of coassociative $4$-planes) carries the tautological vector bundles \\
\centerline{\xymatrix{ \R^3 \oplus \R^4  \ar[r] & \mcv \oplus \mcw \ar[d] \\ & \GG/\SO(4)}}\\
which can be pulled back to any surface immersed in $\R^7$ using its extended Gauss map.  These vector bundles come equipped with natural Riemannian metrics and compatible connections and so their pullbacks also carry that structure. 

\begin{defn}
A smooth immersion $\phi:M^2 \to \R^7$ of a surface is $\GG/\SO(4)${\bf -abelian} if $\Gamma^*(\mcw)$ is flat (and consequently so is $\Gamma^*(\mcv)$).
\end{defn}

To make a clean statement of the geometric correspondence we introduce a nondegeneracy condition on surfaces.
\begin{defn}
A surface $\phi:M \to \R^7$ is $\GG/\SO(4)${\bf -nondegenerate} if its extended Gauss map is an immersion.
\end{defn} 

We will see shortly that any simply connected nondegenerate $\GG/\SO(4)$-abelian surface is equivalent to a solution of an affine extension of a (hyperbolic) integrable system associated to $\GG/\SO(4)$.  In turn, solutions to the integrable system are known to correspond to smooth maps into the twisted loop group $\mcl^{\tau,\sigma}_+(\GGC)$.  The dressing action \cite{Pressley1986} allows one to generate non-trivial smooth maps into  $\mcl^{\tau,\sigma}_+(\GGC)$ (and thus nontrivial solutions and submanifolds) by using elements of the rational loop group $\rlgmtw(\GGC,\C^7)$.   

For the elliptic integrable systems associated to a symmetric space, i.e.~ the harmonic and primitive map systems, having generators for the twisted rational loop group has allowed one to understand the moduli space of solutions when the domain is an $\s{2}$ and $\G^{\tau}=\Un(n)$ \cite{Uhlenbeck1989,Dai2007}.  No such theorem is known for the hyperbolic systems, but having generators certainly allows one to generate arbitrarily complicated solutions. 

To every Riemannian symmetric space there is an associated hyperbolic integrable system of PDE \cite{Terng2005}.   Following \cite{Terng2006} we use an affine extension of the $\GG/\SO(4)$-system.  We will display the extended system, the $\frac{\GG \ltimes \R^{7}}{\SO(4) \ltimes \R^3}$-system, in terms of a $1$-parameter family of flat connections. 

Using the explicit description of $\GG \subset  \SO(7)$ outlined in Appendix \ref{sec:g2}, the Lie algebra of the affine group $\GG \ltimes \R^7$ can be written in block form as
\begin{align*}
\g_2 \ltimes \R^{7}=&\left\{  \left. \bp C&- \trp{B}&X\\B&\nu(C)&Y\\0&0&0 \ep \right|  X \in \R^4, Y\in \R^3, C \in \so(4), B \in \lp   \right\}
\end{align*}
where $\lp$ and the homomorphism $\nu:\so(4) \to \so(3)$ are defined in Appendix \ref{sec:g2}.  Let $\G=\GGC \ltimes \C^7$ and let $\g$ be its Lie algebra.  Define the (commuting) automorphisms
\begin{align*}
&\hat \tau, \hat \sigma: \g \to \g\\
&\hat \tau (\xi)=\ol{\xi}\\
&\hat \sigma (\xi)=S\xi S
\end{align*}
where 
\be
S=\bp -I_4&0&0\\0&I_3&0\\0&0&1 \ep.
\ee
These automorphisms are affine extensions of the automorphisms $\tau$ and $\sigma$ defined in Section \ref{sec:twistedg2}.  They are chosen so that 
\begin{align*}
\g^{ \hat \tau}&=\g_2\ltimes \R^{7}\\
\g^{ \hat \tau,\hat \sigma}&=\lk 
\end{align*}
where
\be
\hat \lk:=\so(4) \ltimes \R^3=\left\{ \left. \bp C&0&0\\0&\nu(C)&Y\\0&0&0 \ep \right| C \in \so(4), Y \in \R^3 \right\}.
\ee
On $\g^{\hat \tau}$ the $+1$-eigenspace of $\hat \sigma$ is $\lk$ and the $-1$-eigenspace is 
\be
\hat \lp=\left\{ \left. \bp 0&-\trp{B}&X\\B&0&0\\0&0&0 \ep \right| B \in \lp; X \in \R^4 \right\}.
\ee
Thus the automorphisms $\hat \tau$ and $\hat \sigma$ on $\G$ define the homogeneous space $\frac{G^{\hat \tau}}{G^{\hat \tau \hat \sigma}}=\frac{\GG \ltimes \R^{7}}{\SO(4) \ltimes \R^3}$, which is an `affine extension' of the symmetric space $\GG/\SO(4)$.

Use the elements
\be
a_1=\bp 1&0&0&0\\0&0&0&0\\0&0&1&0 \ep ,\;\; a_2=\bp 0&0&0&0\\0&0&0&1\\0&0&1&0  \ep
\ee
and the inclusion
\be
a_i \mapsto b_i=\bp 0&-\trp{a_i}&0\\a_i&0&0\\0&0&0\ep 
\ee
to define a torus $\mathfrak{ a} \subset \hat \lp$.

Let $\R^{2}$ have the standard coordinates $(x_1,x_2)$.  A (hyperbolic) integrable system associated  to $\frac{\GG \ltimes \R^{7}}{\SO(4) \ltimes \R^3}$ can now be defined as a family of flat connections of the form 
\be\label{eq:flatformg2}
\theta_{\lambda}=\sum_{i=1}^{2}(\lambda a_i   + [b_i,v])\ed x_i 
\ee
where $v:\R^2 \to \hat \lp $.  This implies that the connection is of the form
\be
\theta_{\lambda}=\bp \al& -\la \trp{\beta} &0 \\ \la \beta & \nu(\al) & \om\\ 0&0&0 \ep
\ee
where $\beta$ is of the form
\be
\bp \beta_1&0&0&0\\0&0&0&\beta_2\\0&0&\beta_1+\beta_2&0 \ep .
\ee

\begin{prop}
A simply-connected nondegenerate $\GG/\SO(4)$-abelian surface $\phi:\R^2 \to \R^7$ is equivalent to a family of flat connections $\theta_{\lambda}$ on $\R^7$ of the form given in Equation \eqref{eq:flatformg2}.  
\end{prop}
\begin{proof}
By a $\GG$-adapted frame of a surface we mean a frame such that $e_5,e_6,e_7$ span the unique extension of the tangent bundle to a bundle of associative three-planes, and  $(e_1, \ldots,e_7) \in \GG$.  As the tangent bundle only plays a supporting role in this geometry, it will be useful for the calculations below to not insist that any one of $e_5,e_6,e_7$ is normal to the surface.  Instead we reserve that gauge freedom in order to find coordinates that are `line of curvature coordinates' with respect to the second fundamental form of the bundle $\phi^*\mcv$.   

A simply-connected surface in $\R^7$ with a $\GG$-adapted frame is equivalent to a surface with a flat $\g_2 \ltimes \R^7$-valued one-form.  We prove that the geometric condition of being $\GG/\SO(4)$-abelian implies that the connection takes on the special form depicted in equation \eqref{eq:flatformg2}.  It is an easy exercise with moving frame arguments to check that the surface defined by the one-form in equation \eqref{eq:flatformg2} is in fact $\GG/\SO(4)$ abelian, so we leave it out.

Let $\theta$ be the pullback of the Maurer-Cartan form to a $\GG/\SO(4)$-abelian surface.  Using the fact that $\GG$ acts transitively on the Grassmannian of associative $3$-planes, we can use the gauge group (i.e.~ adapt coframes) to put the flat connection $\theta$ in the form
\be
\theta=\bp \al& - \trp{\beta} &0 \\  \beta & \nu(\al) & \om\\ 0&0&0 \ep.
\ee
The remaining gauge group is $\SO(4)$ because that is the stabilizer of an associative $3$-plane.  The flatness of $\theta$ implies that the structure equations for the tautological vector bundles $\phi^*\mcv$ and $\phi^*\mcw$ are
\begin{align*}
\ed \alpha + \alpha \w \alpha-\trp{\beta}  \w \beta &=0\\
\ed \nu(\alpha) + \nu(\alpha) \w \nu(\alpha)-{\beta}  \w \trp{\beta} &=0.
\end{align*}
The flatness of $\phi^*(\mcw)$ implies that 
\be
\trp{\beta}  \w \beta =0.
\ee
Because $\SO(4)$ acts transitively on the maximal abelian subalgebras of $\p$, the $\SO(4)$-gauge action can be used to find a coframe in which 
\be
\beta=\bp \beta_1&0&0&0\\0&0&0&\beta_2\\0&0&\beta_1+\beta_2&0 \ep. 
\ee
The sparseness of $\beta$ reduces the structure equation
\be
\ed \beta+\nu(\alpha) \w \beta+\beta \w \alpha =0
\ee
to
\begin{align*}
\ed \beta_1=\ed \beta_2&=0\\
\beta_{11}\w \alpha_{12}&=0\\
\beta_{24}\w \alpha_{42}&=0\\
\beta_{33}\w \alpha_{32}&=0\\
\beta_{11}\w \alpha_{13}+\nu(\alpha)_{13} \w \beta_{33}&=0\\
\beta_{11}\w \alpha_{14}+\nu(\alpha)_{12} \w \beta_{24}&=0\\
\beta_{24}\w \alpha_{41}+\nu(\alpha)_{21} \w \beta_{11}&=0\\
\beta_{24}\w \alpha_{43}+\nu(\alpha)_{23} \w \beta_{33}&=0\\
\beta_{33}\w \alpha_{31}+\nu(\alpha)_{31} \w \beta_{11}&=0\\
\beta_{33}\w \alpha_{34}+\nu(\alpha)_{32} \w \beta_{24}&=0.
\end{align*}
The first equation implies the existence of smooth functions $s,t:\R^2 \to \R$ for which $\beta_1=\ed s$ and $\beta_2=\ed t$.  The next three equations then imply that 
\begin{align*}
\alpha_{21}&=p_1\ed s \\
\alpha_{42}&=p_2\ed t\\
\alpha_{32}&=p_3(\ed s + \ed t).
\end{align*}
The remaining equations reduce to
\begin{align*}
\beta_{1}\w \alpha_{13}+(\alpha_{31}-\alpha_{42}) \w (\beta_{1}+\beta_{2})&=0\\
\beta_{2}\w \alpha_{43}-(\alpha_{21}+\alpha_{43}) \w (\beta_{1}+\beta_{2})&=0\\
\beta_{1}\w \alpha_{14}-(\alpha_{41}+\alpha_{32}) \w \beta_{2}&=0\\
\beta_{2}\w \alpha_{41}+(\alpha_{41}+\alpha_{32}) \w \beta_{1}&=0\\
(\beta_{1}+\beta_{2})\w \alpha_{31}-(\alpha_{31}-\alpha_{42}) \w \beta_{1}&=0\\
(\beta_{1}+\beta_{2})\w \alpha_{34}+(\alpha_{21}+\alpha_{43}) \w \beta_{2}&=0
\end{align*}
The first, second, and third equations imply that 
\begin{align*}
\alpha_{31}&=(2q_1-p_2)\ed s + q_1\ed t\\
\alpha_{41}&=-(q_2+p_3)\ed s + q_2\ed t\\
\alpha_{43}&=q_3 \ed s + (2q_3+p_1)\ed t\\
\end{align*}
for some smooth functions $q_1,q_2,q_3:\R^2 \to \R$ and the last three equations are then automatically satisfied.  

The structure equation for the (vanishing) coframes of the bundle $\phi^*\mcw$  implies that $\nu(\alpha) \w \om=0$, so that
\begin{align*}
\om_1&=r_1\ed s \\
\om_2&=r_2\ed t \\
\om_3&=r_3(\ed s + \ed t)
\end{align*} 
for smooth functions $r_1,r_2,r_3:\R^2 \to \R$.

If we set $x_1=s$ and $x_2=t$ then it follows that $\theta=\sum_{i=1}^{2}( b_i   + [b_i,v])\ed x_i $ for $v=\bp 0&-\trp B & X  \\ B & 0 & 0 \ep$, 
\be
B=\bp 0&p_1&q_1-p_2&-(p_3+q_2) \\ -q_2&-p_2&-(p_2+q_3)&0 \\ -q_1&-p_3&0&q_3\ep,
\ee
and $ X=\trp \bp  r_1&0&r_3 &r_2 \ep$.  In fact, the structure equations directly imply that $\theta_{\la}=\sum_{i=1}^{2}(\lambda b_i   + [b_i,v])\ed x_i $ is flat for any $\la \in \C$.  Thus we have produced a family of flat connections of the desired form.
\end{proof}

As discussed in \cite{Fox2008}, whenever $G^\tau/K$ is a Grassmannian for the fundamental representation $V$ of $G^\tau$, one can define a submanifold geometry $\R^r \to V$ (where $r=\rank(G^\tau/K)$) using an extended Gauss map to insist that the pull back from $G^\tau/K$ of certain tautological bundles are flat.  It is not clear what the submanifold geometry will be when $G^\tau/K$ is not a Grassmannian, e.g.~ $\SO(2n)/\Un(n)$, but it would be interesting to know. 
\appendix
\section{An explicit description of $\GG \subset \SO(7)$}\label{sec:g2}
Let $\Oc$ denote the octonions, the unique real $8$-dimensional division algebra, equipped with the natural metric $\langle x,y \rangle=\Re(x\cdot \bar{y})=\frac{1}{2}(x\cdot \bar{y}+y \cdot \ol{x})$. The compact simple Lie group $\GG$ is known to be the automorphism group of $\Oc$. Since the metric is defined via the multiplication, we get $\GG\subset \SO(\Oc)$. The subspace $\R\cdot 1\subset \Oc$ is fixed by $\GG$, so if we identify $\R^7=\Im(\Oc)$, we obtain the fundamental representation $\GG\subset \SO(7)$.

Let $1,e_1,\ldots,e_7$ be the standard orthonormal basis of $\R^8 \cong \Oc$.  Our convention is that they satisfy the multiplication table:
\begin{center}
\begin{tabular}{c||c|c|c|c|c|c|c|c}
&$1$&$e_1$&$e_2$&$e_3$&$e_4$&$e_5$&$e_6$&$e_7$\\
\hline\hline
$1$&$1$&$e_1$&$e_2$&$e_3$&$e_4$&$e_5$&$e_6$&$e_7$\\
\hline
$e_1$&$e_1$&$-1$&$-e_5$&$-e_6$&$-e_7$&$e_2$&$e_3$&$e_4$\\
\hline
$e_2$&$e_2$&$e_5$&$-1$&$-e_7$&$e_6$&$-e_1$&$-e_4$&$e_3$\\
\hline
$e_3$&$e_3$&$e_6$&$e_7$&$-1$&$-e_5$&$e_4$&$-e_1$&$-e_2$\\
\hline
$e_4$&$e_4$&$e_7$&$-e_6$&$e_5$&$-1$&$-e_3$&$e_2$&$-e_1$\\
\hline
$e_5$&$e_5$&$-e_2$&$e_1$&$-e_4$&$e_3$&$-1$&$e_7$&$-e_6$\\
\hline
$e_6$&$e_6$&$-e_3$&$e_4$&$e_1$&$-e_2$&$-e_7$&$-1$&$e_5$\\
\hline
$e_7$&$e_7$&$-e_4$&$-e_3$&$e_2$&$e_1$&$e_6$&$-e_5$&$-1$
\end{tabular}
\end{center}

From the multiplication table one can calculate that a matrix $X \in \so(7)$ is in $\g_2$ if and only if it satisfies
\be\label{eqn:g2rel}
\begin{split}
X_{67}&= X_{12}+X_{34}\\
X_{75}&= X_{13}+X_{42}\\
X_{56}&= X_{14}+X_{23}\\
X_{51}&+X_{64}-X_{73}=0\\
X_{52}&+X_{63}+X_{74}=0\\
X_{53}&- X_{62}+X_{71}=0\\
X_{54}&- X_{61}-X_{72}=0.
\end{split}
\ee

Let $S=\bp -I_4&0\\0&I_3\ep$ and $\sigma(X)=SXS^{-1}$.  This induces the Cartan decomposition $\g_2=\so(4) \oplus \lp$.  Adapting the description of $\g_2$ to this decomposition we can write an element as
\be
X=\bp A & -\trp{B} \\ B & \nu(A) \ep
\ee
with $A \in \so(4)$, $\nu:\so(4) \to \so(3)$ and $B \in \lp$.  The symmetries of $\lp$ and the map $\nu$ are explicitly defined by \eqref{eqn:g2rel}


\newcommand{\noopsort}[1]{}

\end{document}